\documentclass[a4paper,11pt]{amsproc}

\usepackage{xargs}
\usepackage[dvipsnames]{xcolor}
\usepackage{mathtools, todonotes}
\usepackage{amsmath,amssymb,amsfonts,amsthm}
\usepackage{thmtools}

\definecolor{citegreen}{rgb}{0,0.5,0.15}
\definecolor{linkblue}{rgb}{0,0.4,1}
\definecolor{citebordergreen}{rgb}{0,0.9,0.37}
\usepackage[colorlinks,linkbordercolor={0 0 1}, linkcolor=blue, unicode=true, citecolor=citegreen, linkbordercolor={0 0.4 1}, citebordercolor={0 0.9 0.37}, backref]{hyperref}

\usepackage[alphabetic]{amsrefs}
\usepackage{enumerate, tikz-cd, nicefrac}

\usepackage{breqn}
\usepackage{mathrsfs}
\newcommand{\ignore}[1]{}

\title{On non-isomorphic universal sofic groups}

\author{Vadim Alekseev}
\address{Vadim Alekseev, TU Dresden, 01062 Dresden, Germany}
\email{vadim.alekseev@tu-dresden.de}

\author{Andreas Thom}
\address{Andreas Thom, TU Dresden, 01062 Dresden, Germany}
\email{andreas.thom@tu-dresden.de}


\theoremstyle{plain}
\newtheorem{theorem}{Theorem}[section]
\newtheorem{definition}[theorem]{Definition}
\newtheorem{proposition}[theorem]{Proposition}
\newtheorem{lemma}[theorem]{Lemma}
\newtheorem{corollary}[theorem]{Corollary}

\newtheorem{question}[theorem]{Question}

\theoremstyle{definition}
\newtheorem{remark}[theorem]{Remark}

\newcommand{\beq}{\begin{equation}}
\newcommand{\eeq}{\end{equation}}
\newcommand{\beqn}{\begin{equation*}}
\newcommand{\eeqn}{\end{equation*}}
\newcommand{\brq}{\begin{dmath}[compact]}
\newcommand{\erq}{\end{dmath}}
\newcommand{\brqn}{\begin{dmath*}[compact]}
\newcommand{\erqn}{\end{dmath*}}
\newcommand{\bag}{\begin{align}}
\newcommand{\eag}{\end{align}}
\newcommand{\bagn}{\begin{align*}}
\newcommand{\eagn}{\end{align*}}

\newcommand{\mc}{\mathcal}
\newcommand{\mb}{\mathbb}

\newcommand{\eps}{\varepsilon}
\newcommand{\alg}{\mathrm{alg}}
\newcommand{\met}{\mathrm{met}}

\newcommand{\lacts}{\curvearrowright}

\DeclareMathOperator{\Sym}{Sym}
\DeclareMathOperator{\Alt}{Alt}

\DeclareMathOperator{\Cnt}{C}
\DeclareMathOperator{\Znt}{Z}
\DeclareMathOperator{\Rad}{Rad}

\newcommand{\vertiii}[1]{{|\kern-0.2ex|\kern-0.2ex| #1
    |\kern-0.2ex|\kern-0.2ex|}}

\begin{document}
\begin{abstract}
We show that there are $2^{\aleph_0}$ non-isomorphic universal sofic groups. This proves a conjecture of Simon Thomas.
\end{abstract}

\maketitle

\tableofcontents

\section{Introduction}

Starting with the work of Gromov \cite{MR1694588} and Weiss \cite{MR1803462}, there has been growing interest in the class of sofic groups over the last two decades. Elek and Szab\'{o} \cite{MR2178069} were the first to characterize countable sofic groups by the property of embedability into a metric ultraproduct of symmetric groups with respect to the normalized Hamming metric; following their work these metric ultraproducts were termed \emph{universal sofic groups}. See Pestov's excellent survey for more background \cite{MR2460675}. Metric ultraproducts of finite groups and related constructions have been a subject of intense study eversince, see \cites{MR2298607, MR2607888, MR3710760, MR3846318, MR3210125, MR3162821, MR4292938, MR4232187, MR4141378, MR3966829, MR3749196, MR3299505}. It is elementary to see that the embedability of countable sofic groups does not depend on the particular choice of metric ultraproduct, i.e., on the choice of the ultrafilter. Thus, universal sofic groups cannot be distinguished by their countable subgroups. However, Thomas clarified in \cite{MR2607888} that assuming the negation of the continuum hypothesis, there must be $2^{2^{\aleph_0}}$ non-isomorphic universal sofic groups and naturally raised the question, how the situation unfolds under the assumption of the continuum hypothesis.

Our main result, which proves a conjecture by Thomas  \cite[Conjecture 1.2]{MR2607888}, is the following theorem:

\begin{theorem} \label{thm:mainintro}
There are at least $2^{\aleph_0}$ pairwise non-isomorphic metric ultraproducts $\prod^{\rm met}_{\mathcal U} ({\rm Sym}(n),d_n)$.
\end{theorem}

The result is independent of the question whether we want to assume the continuum hypothesis or not. However, if we assume the continuum hypothesis, it follows from classical arguments that the bound $2^{\aleph_0}$ is tight.

Phrasing our result in slightly different language, we prove that the FO-theory in the sense of model theory for metric structures (see \cite{benyaacov}) of the symmetric group $({\rm Sym}(n),d_n)$ endowed with the normalized Hamming metric
$$d_n(\sigma,\tau) := \frac{|\{1 \leq i \leq n | \sigma(i) \neq \tau(i)\}|}{n}$$does not converge as $n$ tends to infinity. This answers \cite[Question 1.3]{MR2607888} and raises the natural question of the lowest complexity of the FO-sentence (in the sense of model theory for metric structures), that is able to distinguish two sequences of symmetric groups.

Our proof of \autoref{thm:mainintro} uses a blend of recent results revolving around groups with Kazhdan's property (T), see \cite{kazhdan-T}, starting with recent breakthrough results of Caprace--Kassabov \cite{caprace2023tame} and Bartholdi--Kassabov \cite{bartholdi2023property} in providing Kazhdan groups that surject onto a sufficiently large set of alternating groups. We use Kun's analysis of sofic approximation of Kazhdan groups, see \cite{MR4555893, kun2021expanders}, and the application of these results to the study of almost centralizers of sofic approximations of Kazhdan groups obtained by Kun and the second author in \cite{kun2019inapproximability}. Finally, we rely on recent work of Becker and Chapman \cite{MR4634678}, who answered a question from \cite{kun2019inapproximability} and proved  stability of uniform almost homomorphisms from finite groups to permutation groups. Finally, we use more classical work of Felgner \cite{MR1107758} and Wilson \cite{MR1477188} in giving characterizations of finite non-abelian simple groups in terms of first order sentences in the language of group theory.

Let's outline the argument: The basic mechanism of the proof is to observe that the metric groups $({\rm Sym}(n!\!),d_{n!})$ behave in different ways in the sense of model theory for metric structures for different $n \in \mathbb N.$ This is due to the existence of the left-right action of ${\rm Sym}(n) \times {\rm Sym}(n)$ on the set ${\rm Sym}(n)$, where the factors are centralizers of each other and each of them acts freely. In the presence of sufficient expansion, this situation is very rigid and can be uniquely recovered in a suitable metric ultraproduct $\prod^{\rm met}_{\mathcal U} ({\rm Sym}(n!\!),d_{n!})$. Thus, the \emph{algebraic} ultraproduct $\prod^{\rm alg}_{\mathcal U} {\rm Sym}(n)$ can be defined inside the metric ultraproduct $\prod^{\rm met}_{\mathcal U} ({\rm Sym}(n!\!),d_{n!})$ as the centralizer of a particular finite set of elements. In order to be able to make uniform statements, we need the construction of Bartholdi--Kassabov mentioned above. Now, it is folklore how to distingish algebraic ultraproducts of symmetric groups using first order sentences in the language of group theory and this finishes the proof. The actual work and novelty of the arguments lies in the effort to prove the required rigidity of centralizers and double centralizers. As explained above, this makes crucial use of particular expansion generators of ${\rm Sym}(n)$, rigidity results of Kun and the second author, and stability of uniform almost homomorphisms.

\section{Almost homomorphisms}

\subsection{Local almost homomorphisms and sofic approximations}

We will be very brief on the definition of a sofic group and refer to \cite{MR2460675} for details, examples, and references.

\begin{definition}
Let $\Gamma$ be a group. Let $F \subset \Gamma$ be finite and $\delta>0$. We say that a map $\sigma \colon \Gamma \to {\rm Sym}(n)$ is a $(F,\delta)$-homomorphism if $d_n(\sigma(gh),\sigma(g)\sigma(h)) \leq \delta$ for all $g,h \in F$. We say that $\sigma$ is $(F,\delta)$-injective if $d_n(\sigma(g),\sigma(h)) \geq 1 - \delta$ for all $g,h \in F$ with $g \neq h.$
\end{definition}

\begin{definition}\label{def:sofic}
A group $\Gamma$ is sofic if and only if for any finite set $F \subset \Gamma$ and $\delta>0$, $\Gamma$ admits an $(F,\delta)$-injective $(F,\delta)$-homomorphism to ${\rm Sym}(n)$ for some $n$.
\end{definition}

In case of countable groups, which is the most interesting case, it is easy to see that we can characterize soficity by requiring the existence of a sequence of maps $(\sigma_n)_n$ that are $(F,\delta)$-injective $(F,\delta)$-homomorphisms for larger and larger $F$ and smaller and smaller $\delta>0$.

It is often convenient and useful to express asymptotic properties of a sequence of $(F,\delta)$-homomorphisms for increasing $F$ and decreasing $\delta$ in terms of an asymptotic object, a metric ultra-product, that is defined using a non-principal ultrafilter. For background on ultrafilters on $\mathbb N$ see standard references, e.g. \cite{Jech03}. We will denote the set of ultrafilters by $\beta \mathbb N$, then the set of non-principal ultrafilters equals $\beta \mathbb N \setminus \mathbb N$. {Every map $f\colon \mathbb N \to \mathbb N$ induces a map $f_\ast\colon \beta\mathbb N\to \beta\mathbb N$ by $f_\ast(\mathcal U)=\{A\subseteq\mathbb N\mid f^{-1}(A)\in\mathcal U\}$; if $f(n)\to\infty$, then $f_\ast$ maps non-principal ultrafilters to non-principal ultrafilters.}
We will say that an ultrafilter $\mathcal U$ is \emph{supported} on a set $A \subseteq \mathbb N$ if $A \in \mathcal U$ or equivalently, if $\mathcal U \in \beta A \subseteq \beta \mathbb N.$
Note that every bounded sequence of real numbers has a limit \emph{along} an ultrafilter. In particular, we can study the behaviour of a bounded sequence of real numbers or the sequence of truth values of a fixed first-order formula applied to a sequence of structures along an ultrafilter.

Let $\mathcal U$ be a non-principal ultrafilter on $\mathbb N$ and $G_n$ be groups. We consider the algebraic ultraproduct
\[
\prod{\vphantom{\prod}}^{\alg}_{\mathcal U} G_n \coloneqq  \left(\prod_n G_n \right)/ \left\{(g_n)_n \in \prod_n G_n \,\Big\vert \,
  g_n = 1_n\text{ along }\mc U\right\}.
\]

{At this point we want to recall \L{o}\'{s}'s theorem for algebraic ultraproducts, which states that a first-order formula $\varphi$ in the language of groups holds in $\prod^{\alg}_{\mathcal U} G_n$ if and only if the set of $n \in \mathbb N$ such that $\varphi$ holds in $G_n$ is an element of $\mathcal U$. }

We now define the \emph{universal sofic group} with respect to the ultrafilter $\mathcal U$ to be
\[
\Sym(\mathcal U) \coloneqq \prod\vphantom{\prod}^{\met}_{\mathcal U} ({\rm Sym}(n),d_n) =\left(\prod_n {\rm Sym}(n)\right)/{\rm N}(\mathcal U),
\]
where
\[
{\rm N}(\mathcal U) = \left\{(\sigma_n)_n \in \prod_n {\rm Sym}(n) \,\Big\vert \,
 \lim_{n \to \mathcal U} d_n(1_n,\sigma_n)=0\right\}.
\]
Note that ${\rm Sym}(\mathcal U)$ carries a natural bi-invariant metric $d_{\mathcal U}$ arising as the ultralimit of the metrics on ${\rm Sym}(n).$ This metric takes values in the interval $[0,1]$. {The metric ultraproduct plays a role in model theory of metric structures analogous to the one of algebraic ultraproducts, see \cite{benyaacov}.}
\begin{definition}
A subgroup $\Gamma \leq {\rm Sym}(\mathcal U)$ is said to be $1$-discrete if $d_{\mathcal U}(g,h)=1$ for all $g,h \in \Gamma$ with $g \neq h.$
\end{definition}

The following result goes back to \cite{MR2178069}, where universal sofic groups were considered for the first time and many basic results about them were proved, {see also \cite{MR2607888, MR2460675}.}

\begin{theorem}[Elek--Szab\'{o}] \label{thm:esz}
Let $\Gamma$ be a countable group. The following conditions are equivalent.
\begin{enumerate}[$(i)$]
    \item The group $\Gamma$ is sofic.
    \item The group $\Gamma$ is isomorphic to a $1$-discrete subgroup of ${\rm Sym}(\mathcal U)$ for all non-principal ultrafilters $\mathcal U.$
    \item The group $\Gamma$ is isomorphic to a subgroup of ${\rm Sym}(\mathcal U)$ for some non-principal ultrafilter $\mathcal U.$
\end{enumerate}
\end{theorem}

The basic observation that leads to the previous theorem is that a sequence of maps $\sigma_n \colon \Gamma \to {\rm Sym}(m_n)$, where $\sigma_n$ is a $(F_n,\delta_n)$-injective $(F_n,\delta_n)$-homomorphism, with $\lim_{n \to \infty} \delta_n = 0$ and $(F_n)_n$ a sequence of finite subsets that increases to $\Gamma$, leads naturally to an injective homomorphism
$$\sigma = [\sigma_n]_n \colon \Gamma \to {\rm Sym}(\mathcal U)$$ with $\mathcal U$ an ultrafilter supported on the set $\{m_n \mid n \in \mathbb N\}$. Moreover, the image is clearly $1$-discrete. Conversely, any such $1$-discrete embedding arises from a sequence as above, {and we will use the notion \emph{sofic approximation} both for $\sigma$ and the sequence $(\sigma_n)$ in this context}. We say that two {sofic approximations} $(\sigma_n)_n$ and $(\sigma'_n)_n$ are \emph{essentially  the same} {or \emph{asymptotically equivalent}} (along equivalent ultrafilters $\mathcal U,\mathcal V$, {see \autoref{def:equivalentultrafilters}} and \autoref{thm:equl} for details) if they lead to the same embedding, i.e.,  $\sigma=\sigma'$ under this identification. {Note that this also covers cases when $\mathcal U \neq \mathcal V$, but still canonically ${\rm Sym}(\mathcal U) \cong {\rm Sym}(\mathcal V).$}

As a particular consequence of \autoref{thm:esz}, we note again that existence of countable subgroups cannot distinguish the groups ${\rm Sym}(\mathcal U)$ for different ultrafilters up to isomorphism.

If $\Gamma$ is finitely generated, there is a slightly different picture of sofic approximations {using the language of Schreier graphs. Let $S$ be a finite set. An $S$-labelled Schreier graph is an oriented $S$-labelled graph, such that every vertex has exactly one ingoing and one outgoing edge labelled with $s \in S$. There is a natural correspondence between $S$-labelled Schreier graphs and actions of the free group on $S$. Let $S \subset \Gamma$ be a finite generating set.} Any $(F,\delta)$-injective $(F,\delta)$-homomorphism $\sigma \colon \Gamma \to {\rm Sym}(n)$ can be restricted to $S$ and leads to a Schreier graph ${\rm Sch}(n,\sigma(S))$ of the free group on $S$; {indeed, we put an edge $(i,s,j)$ for $i,j \in \{1,\dots,n\}, s \in S$ if $\sigma(s)(i)=j$.} It is sometimes more convenient to work with sequences of Schreier graphs that look locally more and more like the Cayley graph ${\rm Cay}(\Gamma,S)$ instead of families  of maps as in  \autoref{def:sofic}. We will frequently switch between the different pictures.

\subsection{Uniform almost homomorphisms}

A different notion of almost homomorphism arises when we require the almost multiplicativity to occur in a uniform way on the entire domain of the map.

\begin{definition}
Let $G$ be a group and $\delta>0$. We say that a map $\sigma \colon G \to {\rm Sym}(n)$ is a uniform $\delta$-homomorphism if
$$d_n(\sigma(gh),\sigma(g)\sigma(h))\leq \delta, \quad \forall g,h \in G.$$

\end{definition}

Following standard terminology, we say that two maps $\sigma,\sigma' \colon G \to {\rm Sym}(n)$ are uniformly $\varepsilon$-close, if $d_n(\sigma(g),\sigma'(g))\leq \varepsilon$ for all $g \in G.$ Uniform perturbations of homomorphisms are the \emph{trivial} examples of uniform almost homomorphisms. Uniform almost homomorphisms were first studied by Kazhdan \cite{kazhdan} in a more functional analytic setting. The theme was further developed in \cite{burgerozawathom, hatamigowers, dechiffreozawathom}. The setting of uniform almost homomorphisms with permutation group targets comes up naturally in the study of almost centralizers of sofic approximations of Kazhdan groups, see \cite{kun2019inapproximability}. The most striking rigidity/stability result  for uniform almost homomorphisms in the permutation group setting is the following theorem, which answered a question from \cite{kun2019inapproximability}:

\begin{theorem}[Becker--Chapman, \cite{MR4634678}*{Theorem 1.2}]\label{thm:becker-chapman}
For every $\varepsilon>0$, there exists $\delta>0$, such that for every $n \in \mathbb N$, every finite group $G$ and every uniform $\delta$-homomorphism $\sigma \colon G \to {\rm Sym}(n)$, there exists a homomorphism $\pi \colon G \to {\rm Sym}(m)$ with $m \in [n,(1+\varepsilon)n]$ which is uniformly $\varepsilon$-close to $\sigma \oplus 1_{m-n}.$
\end{theorem}

In fact the proof shows that $\delta=\varepsilon/2039$ is enough.

\section{Applications of Kazhdan's property (T)}

\subsection{Kazhdan groups and their finite quotients}

Already in the work of Thomas \cite{MR2607888}, the existence of expander generators of symmetric groups played a crucial role. There, the original work of Kassabov \cite{kassabov2007} was used in order to obtain a certain rigidity of the centralizer of a certain finite set of elements in a universal sofic group. In our approach, we need to use more refined result in the same direction. Indeed, in the meantime, a question of Lubotzky was answered positively, and the existence of a Kazhdan group with infinitely many alternating quotients was proven. We will need the following consequence of \cite{bartholdi2023property}*{Corollary 14}.

\begin{theorem}[Bartholdi--Kassabov] \label{thm:bartkass}
There exists a Kazhdan group $\Gamma$, which surjects onto ${\rm Alt}(p^4-1)$ for all primes $p\geq 13.$
\end{theorem}

It would be more convenient for the proof of our main theorem to have a Kazhdan group $\Gamma$, which surjects onto all alternating groups in a way such that any non-trivial $g \in \Gamma$ becomes trivialized only in finitely many of these quotients. However, it seems to be out of reach to construct such a group.

\subsection{Sofic approximations of Kazhdan groups}

The work of G\'{a}bor Kun \cite{kun2021expanders}, anwering a question of Bowen, was the first to prove a positive structure result for sofic approximation beyond the amenable case.

\begin{theorem}[Kun] \label{thm:kun}
Every sofic approximation of a Kazhdan group is essentially a disjoint union of expanders.
\end{theorem}

{More precisely, we say that a finite graph $X$ is a $c$-expander if its Cheeger constant, defined by
\[
h(X) \coloneqq \min_{A \subseteq V(X), {0<}|A| \leq |V(X)|/2} \frac{|\partial A|}{|A|},
\]
where $\partial A = E(A,V\setminus A)$ is the edge boundary of $A$, satisfies $h(X)\geq c$. A sequence of finite graphs $(X_n)_n$ is an expander sequence if $|X_n|\to \infty$ and there is a $c>0$ such that $X_n$ is a $c$-expander for all $n$. }

In subsequent work of Kun and the second author \cite{kun2019inapproximability} it was realized that sofic approximations of Kazhdan groups $\Gamma$ admit a special rigidity when it comes to the study of almost automorphisms of such. Let $S$ be a finite {symmetric} generating set for $\Gamma$. For a given $(F,\delta)$-injective $(F,\delta)$-homomorphism $\sigma \colon \Gamma \to {\rm Sym}(n)$ (where without loss of generality $F\supseteq S$) it makes sense to study the set of $\varepsilon$-automorphisms. These are bijections $\rho \in {\rm Sym}(n)$ of the vertex set of the associated Schreier graph ${\rm Sch}(n,\sigma(S))$, such that
{\[\sum_{s\in S}d_n(\rho\sigma(s),\sigma(s)\rho)\leq\varepsilon,\]}
{i.e., $\rho$ fails to preserve at most $\varepsilon n$ labelled edges.}

{If ${\rm Sch}(n,\sigma(S))$ is a $c$-expander, we now claim that two $\varepsilon$-automorphisms $g$, $h$ are either $(2\varepsilon/c)$-close or $(1-2\varepsilon/c)$-apart.  Indeed, let $A$ be the set where $g$ and $h$ agree. If both preserve an $s$-labelled edge $(v,w)$, then $g(w)=\sigma(s)g(v)$ and $h(w)=\sigma(s)h(v)$, so $v\in A$ if and only if $w\in A$. Therefore every edge in $\partial A$ fails to be preserved by at least one of $g,h$, and hence
\[
c\min\{|A|,n-|A|\}\leq |\partial A|\leq 2\varepsilon n.
\]
Since $d_n(g,h)=1-|A|/n$, the claim follows.}

It thus makes sense to talk about the set of clusters of $\varepsilon$-automorphisms. Naturally, one would expect that the set of clusters forms a group. However, the evident problem is that the product of two $\varepsilon$-homomorphisms is a priori only a $2\varepsilon$-homomorphism, so that we run out of our set of clusters unless some form of self-improvement mechanism can be implemented. One of the main technical results in \cite{kun2019inapproximability} is that this is indeed possible in the situation just discussed.

\begin{theorem}[{\cite[see Theorem 3.4, Lemmas 4.1 and 4.2]{kun2019inapproximability}}] \label{thm:kun-thom}
{Let $\Gamma$ be a Kazhdan group with a finite symmetric generating set $S$. There exist $c>0$ and $\varepsilon_0\in(0,c/10)$ with the following property. For every $\eta\in(0,\varepsilon_0)$ there are a finite set $F\supseteq S$ in $\Gamma$ and $\delta>0$ such that, whenever $\sigma\colon\Gamma\to\Sym(n)$ is an $(F,\delta)$-injective $(F,\delta)$-homomorphism and ${\rm Sch}(n,\sigma(S))$ is a $c$-expander, every cluster of $\varepsilon$-automorphisms, for $\eta<\varepsilon<\varepsilon_0$, contains an $\eta/2$-automorphism.

For each such $\varepsilon$, the clusters form a finite group $G$: multiply two clusters by composing $\eta/2$-automorphism representatives and taking the cluster of their product. For $\eta<\varepsilon<\varepsilon'<\varepsilon_0$, inclusion of the sets of almost automorphisms canonically identifies the respective groups. The clusters have diameter at most $2\varepsilon/c$, and representatives of distinct clusters have distance at least $1-2\varepsilon/c$.}
\end{theorem}
{
\begin{proof}[Proof sketch]
The first claim is \cite[Theorem 3.4]{kun2019inapproximability}. This gives the improvement mechanism discussed before the statement of the theorem: two $\eta/2$-automorphisms multiply to an $\eta$-automorphism which again can be corrected to $\eta/2$-automorphism. Thus we deduce that the set of clusters of $\varepsilon$-automorphisms of ${\rm Sch}(n,\sigma(S))$ becomes a finite group $G$ with respect to multiplication of representatives of clusters (see also \cite[Lemma 4.2]{kun2019inapproximability}). Moreover, by the claim above (see also \cite[Lemma 4.1]{kun2019inapproximability}), two $\eps$-automorphisms are either $(2\eps/c)$-close or $(1-2\eps/c)$-apart. Since every $\varepsilon$-automorphism is close to an $\eta/2$-automorphism (and their cluster is well-defined by the claim in the previous sentence), the group $G$ is independent of the choice of $\eps\in (\eta,\eps_0)$. This finishes the proof.
\end{proof}
}

We call the group $G$ obtained from the previous theorem the cluster group of $\varepsilon$-{automorphisms} or the $\varepsilon$-centralizer of $\sigma \colon \Gamma \to {\rm Sym}(n)$. If $G$ is the cluster group of $\varepsilon$-automorphisms as above, we may choose a representative of each cluster and obtain a uniform $C\varepsilon$-homomorphism  $\sigma' \colon G \to {\rm Sym}(n)$ for some $C$ depending only on {the chosen generating set $S$ of $\Gamma$}. This observation will be used in combination with \autoref{thm:becker-chapman} in the proof of the main theorem.

\medskip

We also obtain an immediate corollary. {Recall that the centralizer of a subgroup $H$ of a group $G$ is defined as $\Cnt_G(H) = \{g \in G \mid gh=hg \text{ for all } h \in H\}.$ We will also use the notation $\Cnt(H)$ if the ambient group is clear from the context and denote by $\Cnt_G(F)$ or $\Cnt(F)$ the centralizer of a subset $F$ of a group.}

\begin{corollary}\label{cor:centralizer-is-an-algebraic-ultraproduct}
The centralizer of a sofic approximation of a Kazhdan group by expanders is an algebraic ultraproduct of finite groups. More precisely, let $\Gamma$ be a Kazhdan group, $S \subset \Gamma$ be a finite generating set, and $\sigma = [\sigma_n]_n \colon \Gamma\to \prod_{\mathcal V}^{\rm met} ({\rm Sym}(n),d_n)$ is a sofic approximation by expanders, and $\eps > 0$ sufficiently small. Then the centralizer of $\sigma(\Gamma)$ is $1$-discrete and isomorphic to the algebraic ultraproduct $\prod_{\mathcal V}^{{\alg}} G_n$ of finite groups $G_n$, where $G_n$ is the $\varepsilon$-centralizer of $\sigma_n(S)$.
\end{corollary}

\begin{proof}
By definition of the metric ultraproduct, the centralizer of $\sigma(\Gamma)$ consists of elements $g = [g_n]_n$ such that $g_n$ is an $\eps_n$-automorphism of ${\rm Sch}(n,\sigma_n(S))$ for all sufficiently large $n$, and $\eps_n\to 0$. By \autoref{thm:kun-thom}, for sufficiently large $n$, $g_n\in G_n$, where $G_n$ is the $\varepsilon$-centralizer of $\sigma_n(S)$. Moreover, if $g,h \in \Cnt(\sigma(\Gamma))$ are different, then $d(g_n,h_n)$ is bounded away from zero along $\mc V$ (say, by $\eps'\in (0,\eps_0)$), and therefore $g_n$ and $h_n$ represent different elements of $G_n$ for sufficiently large $n$. Therefore \autoref{thm:kun-thom} implies $d_n(g_n,h_n) \geq 1 - 2\eps'/c$. This shows that the centralizer is $1$-discrete and isomorphic to the algebraic ultraproduct $\prod_{\mathcal V} G_n$.
\end{proof}

% \orng{
% \begin{proof}
% Fix a sufficiently small threshold $\varepsilon$. Diagonalizing the finite approximation requirements in \autoref{thm:kun-thom}, choose $\eta_n\to_{\mathcal V}0$ such that every cluster in $G_n$ has an $\eta_n$-automorphism representative. Choose such representatives $r_n(g)$ for all $g\in G_n$, with $r_n(1)=1$; outside the relevant $\mathcal V$-large sets take $G_n=\{1\}$ and $r_n(1)=1$. We define
% \[
% \prod_{\mathcal V}^{\alg}G_n\longrightarrow\Cnt(\sigma(\Gamma)),
% \qquad [(g_n)_n]\longmapsto[(r_n(g_n))_n].
% \]
% The representatives commute asymptotically with $S$, uniformly over $G_n$. Moreover, $r_n(g)r_n(h)$ and $r_n(gh)$ have errors at most $2\eta_n$ and $\eta_n$, respectively, and belong to the same cluster. The separation estimate therefore makes their distance tend uniformly to zero, so the displayed map is a homomorphism. Representatives of distinct clusters have distance at least $1-2\eta_n/c$, proving injectivity and $1$-discreteness.

% Now let $[h_n]_n\in\Cnt(\sigma(\Gamma))$. Its error tends to zero along $\mathcal V$, so it determines a cluster $g_n\in G_n$ on a $\mathcal V$-large set. Applying the same separation estimate to $h_n$ and $r_n(g_n)$ shows that their distance tends to zero along $\mathcal V$. Thus the map is surjective.
% \end{proof}
% }

Note that the assumption that the sofic approximation is by expanders is crucial for the $1$-discreteness of the centralizer. As it turns out, there is a partial converse to this result: $1$-discreteness of the centralizer and the double centralizer guarantees that the approximation is by expanders.

\begin{lemma}\label{lem:discrete-centralizers-implies-expander}
Let $\sigma \colon \Gamma\to \prod_{\mathcal V}^{\rm met} ({\rm Sym}(n),d_n)$ be a sofic approximation of an infinite Kazhdan group. If the centralizer $\Cnt(\sigma(\Gamma))$ and the double centralizer $\Cnt(\Cnt(\sigma(\Gamma)))$ are $1$-discrete, then the sofic approximation is essentially by expanders, i.e., the Schreier graphs obtained from the sofic approximation are asymptotically equivalent to a $c$-expander sequence, where $c > 0$ only depends on the Kazhdan constant of $\Gamma$.
\end{lemma}
\begin{proof}
By \autoref{thm:kun}, we can assume that the graphs $X_n$ induced by our sofic approximation of $\Gamma$ are disjoint unions of expanders; let $Y_n$ be a component of maximal size in $X_n$.

We claim that $\lim\limits_{n\to \mc U} |Y_{n}|/|X_n| = 1$. Indeed, suppose that $\lim\limits_{n\to \mc U} |Y_{n}|/|X_n| = {\alpha} < 1$. If {$\alpha>0$}, then the sofic approximation decomposes as the direct product { $\sigma = \sigma_1\times \sigma_2$} of its restrictions to $Y_{n}$ and $Y_{n}^c$, and therefore the centralizer $\Cnt(\sigma(\Gamma))$ contains $\Cnt(\sigma_1(\Gamma))\times \Cnt(\sigma_2(\Gamma))$. If $\Cnt(\sigma_1(\Gamma))$ is non-trivial, the elements of the form $(c_1,1)\in \Cnt(\sigma_1(\Gamma))\times \Cnt(\sigma_2(\Gamma))$ have support at most $\alpha < 1$ which contradicts the assumption that $\Cnt(\sigma(\Gamma))$ is $1$-discrete. If $\Cnt(\sigma_1(\Gamma))$ is trivial, then $\Cnt(\Cnt(\sigma_1(\Gamma)))$ contains the ultraproduct of $\Sym(Y_{n})$ which contradicts the assumption that $\Cnt(\Cnt(\sigma(\Gamma)))$ is $1$-discrete. The case $\alpha=0$ is similar. Indeed, now all components are small and we can replace $Y_n$ by a disjoint union of components of size approximately $|X_n|/2$ and repeat the argument above. This finishes the proof.
\end{proof}

\section{The first order theory of finite groups}

We do not claim any novelty for the results in this section. However, we found it convenient to recall some results in detail and write our own proofs of some basic observations.

\subsection{Characterizing alternating groups}

The following result is due to Felgner \cite{MR1107758}, see the work of Wilson \cite[Theorem 5.1]{MR1477188} for a proof.

\begin{theorem}[Felgner] \label{thm:felgner}There exists a FO-sentence $\varphi$ in the language of groups which is satisfied by a finite group if and only if it is non-abelian simple.
\end{theorem}

For convenience, let us spell out a concrete FO-sentence that characterizes simplicity among non-abelian finite groups. Following Wilson's account of Felgner's result, we set: {
$$\varphi_1 = \left(\forall g, h\ (g\neq 1 \wedge \Cnt(g,h) \neq \{1\}) \rightarrow \left( \bigcap_{k \in G} (\Cnt(g,h)\Cnt(\Cnt(g,h)))^k = \{1\} \right)\right)$$
The sentence $\varphi_1$ asserts that for every non-trivial element $g$ and every element $h$ such that the centralizer of $g$ and $h$ is non-trivial, the intersection of all conjugates of the product of the centralizer of $g$ and $h$ and its own centralizer is trivial. This condition is satisfied by all non-abelian finite simple groups, but also by some non-simple groups. To exclude these, we also need to assert that every element is a commutator:
$$\varphi_2 = (\forall g\ \exists h_1,h_2 \ (g=[h_1,h_2])).$$
Then, a nontrivial finite group satisfies $\varphi \coloneqq \varphi_1 \wedge \varphi_2$ if and only if it is non-abelian simple. Back when Felgner and Wilson studied this problem and before the Ore Conjecture (asserting that every finite non-abelian simple would in fact satisfy $\varphi_2$) was proven in \cite{MR2654085}, one had to replace $\varphi_2$ by a more complicated sentence asserting that every element is a product of at most $k$ commutators for some fixed $k$.}

Starting with this FO-sentence, it is possible to characterize more specific families of non-abelian finite simple groups:
\begin{proposition} \label{prp:alternating}
There exists a FO-sentence in the language of groups that is satisfied for a {finite} {non-abelian} simple group if and only if it is alternating.
\end{proposition}
\begin{proof}
By the classification of finite simple groups, every finite simple group is either cyclic of prime order, alternating, of Lie type, {sporadic, or the Tits group}.
Since the abelian finite simple groups are exactly the cyclic groups $C_p$ of prime order, we may exclude them by conjoining the first-order sentence $\exists x\,\exists y\,([x,y]\neq 1)$, and since the sporadic simple groups {together with the Tits group} form a finite list we may also exclude them one by one. So it suffices to treat alternating groups versus groups of Lie type.

We build a sentence that holds in $\Alt(n)$ for all sufficiently large $n$, fails in every finite simple group of Lie type, and then patch the finitely many remaining alternating groups by a finite disjunction.

By a theorem of J.\,S.\ Wilson, there is an $\mathrm{FO}$-formula $\rho(z)$ (in the language of groups) such that for every finite group $H$ and every $h\in H$,
\[
H \models \rho(h)\quad\Longleftrightarrow\quad h\in \Rad(H),
\]
where $\Rad(H)$ denotes the solvable radical of $H$ \cite{wilson-radical}. For a parameter $x$ in a group $G$, the centralizer $\Cnt_G(x)$ is definable by the formula $[u,x]=1$. Relativizing $\rho$ to $\Cnt_G(x)$ yields a formula $\rho_x(z)$ such that, for finite $G$,
\[
G \models \rho_x(g)\quad\Longleftrightarrow\quad g\in \Rad(\Cnt_G(x)).
\]
For each prime $p$, define a first-order {sentence} $\psi_p$ as follows. The sentence $\psi_p$ asserts that there exist elements $x,a,b$ {and a subgroup $K\leq\Cnt_G(x)$ isomorphic to $\Alt(5)$} such that
\[
{x^p=1},\ x\neq 1,\ \rho_x(a),\ {a^p=1},\ a\neq 1,\ {b\in K},\ \text{and }[a,b]\neq 1.
\]
Equivalently: there is an element of order $p$ such that the solvable radical of $\Cnt_G(x)$ contains {an element $a$ of order $p$ which is not centralized by some subgroup $K\cong\Alt(5)$ of $\Cnt_G(x)$}.
{The subgroup $K$ is specified by $60$ distinct elements satisfying the multiplication table of $\Alt(5)$, so all these conditions are first-order.}

Our first claim is that $\Alt(n)$ satisfies $\psi_2$ and $\psi_3$ for $n\geq{15}$. Indeed, set
\[
x=(1\,2\,3)(4\,5\,6){(7\,8\,9)(10\,11\,12)(13\,14\,15)}\in\Alt(n).
\]
Then $\Cnt_G(x)$ contains the abelian normal subgroup
\[
A=\langle(1\,2\,3),(4\,5\,6){,(7\,8\,9),(10\,11\,12),(13\,14\,15)}\rangle\cong{C_3^5},
\]
hence $A\leq\Rad(\Cnt_G(x))$ and so $\rho_x(a)$ holds for $a=(1\,2\,3)$. {Here we use the usual wreath-product description of permutation centralizers, see \cite[(4.1.19)]{james-kerber}. Let $K\cong\Alt(5)$ permute the five ordered triples by even permutations and fix all remaining points.} Moreover,
\[
{b=(1\,4\,7)(2\,5\,8)(3\,6\,9)\in K}
\]
centralizes $x$, but conjugates $a$ to $(4\,5\,6)\neq a$, so $[a,b]\neq 1$. Thus $\Alt(n)\models\psi_3$ for $n\geq{15}$.

Now, set
\[
x=(1\,2)(3\,4)(5\,6)(7\,8){(9\,10)(11\,12)}\in\Alt(n).
\]
{Let $V\leq\Alt(n)$ consist of the products of an even number of these six transpositions.}
Then $V\cong{C_2^5}$ is a normal $2$-subgroup of $\Cnt_G(x)$ (it is the base group inside the natural wreath-product part of $\Cnt_G(x)$, {intersected with $\Alt(n)$}), hence $V\leq\Rad(\Cnt_G(x))$ and so $\rho_x(a)$ holds for $a=(1\,2){(11\,12)}\in V$.
{Let $K\cong\Alt(5)$ permute the first five ordered pairs by even permutations and fix all other points.} Finally,
\[
{b=(1\,3\,5)(2\,4\,6)\in K}
\]
centralizes $x$, but conjugates $a={(1\,2)(11\,12)}$ to ${(3\,4)(11\,12)}\neq a$, hence $[a,b]\neq 1$.
Thus $\Alt(n)\models\psi_2$ for $n\geq{15}$. Consequently,
\[
\Alt(n)\models(\psi_2\wedge\psi_3)\qquad(n\geq{15}).
\]

Let $G$ be a finite simple group of Lie type over $\mathbb F_q$ of defining characteristic $p$ (twisted or untwisted).
We are going to show that $G$ does not satisfy $\psi_\ell$ if $\ell$ is a prime different from $p$. Indeed, if $x\in G$ has prime order $\ell\neq p$, then $x$ is semisimple. {Put $C=\Cnt_G(x)$. The structure of semisimple centralizers gives a normal subgroup $E\lhd C$, a possibly trivial central product of quasisimple groups, such that $C/E$ is solvable; see \cite[Theorem 14.2, Proposition 14.20, Theorem 24.17 and Remark 24.18]{malle-testerman}. Indeed, the connected algebraic centralizer is reductive and its component group is abelian. Passing to simply connected covers of its simple factors gives finite groups that are quasisimple apart from the small exceptions, which are solvable or have a simple normal subgroup with solvable quotient. Toral factors and the kernels and cokernels on fixed points of these central isogenies are abelian. Here a group is quasisimple if it is perfect and its quotient by its center is non-abelian simple.

Now $\Rad(C)\cap E\leq\Znt(E)$, and since $E$ is perfect, it follows that
\[
[\Rad(C),E]=1.
\]
Every subgroup $K\cong\Alt(5)$ of $C$ is perfect, so its image in the solvable group $C/E$ is trivial. Hence $K\leq E$ and $[\Rad(C),K]=1$. Therefore the commutator condition in $\psi_\ell$ cannot be met.}
Combining the previous observations, we conclude that if a non-abelian finite simple group {is neither sporadic nor the Tits group} and satisfies $\psi_2\wedge\psi_3$, then it is an alternating group. Finally, we can account for $\Alt(n)$ for $n\in{\{5,\dots,14\}}$ one by one. This concludes the proof.
\end{proof}

\begin{proposition}
\label{thm:congr} Let $q \geq 3$ be a prime number and $l \in \mathbb Z$ with $0 \leq l \leq q-1.$
There exists a FO-sentence $\varphi(l,q)$ in the language of groups such that a finite group $G$ satisfies $\varphi(l,q)$ if and only if $G$ is isomorphic to ${\rm Alt}(n)$ with $n \equiv l \mod q$ {and $n \geq 5$}.
\end{proposition}

\begin{proof}
{By \autoref{thm:felgner} and \autoref{prp:alternating}, we may already assume that $G$ is an alternating group and $n \geq 5$. Let $l_0 \ge 5$ be the smallest integer with $l_0\equiv l\pmod q$. We define $\varphi(l,q)$ as the sentence
saying: there exist ${g}$ and a subgroup $A\le \Cnt_G(g)$ such that (i) $g^q=1$,
(ii) $A\cong \Alt(l_0)$, (iii) $[\Cnt_G(g):A\cdot \Cnt_{\Cnt_G(g)}(A)]\le 2$.

Formally, (ii) is expressed as usual by existentially quantifying $m=|\Alt(l_0)|$ distinct elements
$a_1,\dots,a_m$ in $G$ and asserting that they satisfy the multiplication table of $\Alt(l_0)$; the set $\{a_1,\dots,a_m\}$ is then the subgroup $A$.
The condition $A\le \Cnt_G(g)$ is $\bigwedge_i [a_i,g]=1$.
The centralizer $\Cnt_{\Cnt_G(g)}(A)$ is definable by $[x,g]=1\wedge \bigwedge_i [x,a_i]=1$,
and membership in $A\cdot \Cnt_{\Cnt_G(g)}(A)$ is definable by an $\exists$-formula.
The index bound $\le 2$ is expressible by the usual ``at most two cosets'' sentence.

We need to prove two implications. Assume $G\cong \Alt(n)$ with $n\equiv l\pmod q$.
Then $n=rq+{l_0}$ for some $r\ge 0$. Choose $g\in \Alt(n)$ as a product of $r$ disjoint $q$-cycles and ${l_0}$ fixed points.
In $\Sym(n)$ one has the standard centralizer description
\[
\Cnt_{\Sym(n)}(g)\ \cong\ (C_q\wr S_r)\times S_{{l_0}}.
\]
Intersecting with $\Alt(n)$ changes index by at most $2$.
Let $A\cong \Alt(l_0)$ be the alternating group on the $l_0$ fixed points. Then $A\le \Cnt_{\Alt(n)}(g)$,
$A$ commutes with the rest of the centralizer, and
$[\Cnt_{\Alt(n)}(g):A\cdot \Cnt_{\Cnt_{\Alt(n)}(g)}(A)]\le 2$.
Hence, we obtain $\Alt(n)\models \varphi(l,q)$.

Conversely, assume $G\models \varphi(l,q)$.
Fix witnesses $g$ and $A\cong \Alt(l_0)$ for $\varphi(l,q)$ and set $C=\Cnt_G(g)$, $K=A\cdot \Cnt_C(A)$.
{Since $A$ is centerless, $K=A\times\Cnt_C(A)$; moreover, $[C:K]\leq2$ implies $K\lhd C$. Write the cycle type of $g$ as $q^r1^m$, so $n=rq+m$. The base group $Q=C_q^r$ of the wreath-product part of $C$ is abelian and normal in $C$. Since $q$ is odd, $Q\leq K$, and its projection to the simple factor $A$ is trivial. Thus $A$ centralizes $Q$. The centralizer calculation gives
\[
\Cnt_C(Q)=Q\times\Alt(m).
\]
As $A$ is perfect, its projection to $Q$ is trivial, so $A\leq\Alt(m)$ and $m\geq5$. Now $K\cap\Alt(m)$ has index at most two in the simple group $\Alt(m)$; hence $\Alt(m)\leq K$. Since $A\lhd K$, we obtain $A\lhd\Alt(m)$ and therefore $A=\Alt(m)$. Comparing orders gives $m=l_0$, and consequently $n\equiv l\pmod q$, as required.}}
\end{proof}

Results like the previous proposition are {probably} well-known {to experts}, see for example the proof of Theorem 3.1 in \cite{thomasetal}. It is natural to wonder what kind of subsets of the natural numbers can be defined in this way. More precisely:

\begin{question}
For which subsets $A \subset  \mathbb N$ does there exist a FO-sentence $\varphi$ in the language of groups such that
$n \in A$ if and only if ${\rm Alt}(n)$ satisfies $\varphi$?
\end{question}

\begin{remark}
As an example, note that the sentence
$$\exists g\  \forall h \ (gh=hg) \rightarrow (h=1 \vee (\exists k \ hk=kg))$$
characterizes ${\rm Sym}(n)$ with $n=p$ or $n=p+1$ and $p$ a prime within the set of all symmetric groups. Hence, more exotic sets than mere arithmetic progressions are definable using FO-sentences in the language of groups. {See \cite{MO302743} for a discussion triggered by a question of M. Sapir.}
\end{remark}

\subsection{Non-isomorphic ultraproducts of alternating groups}

The following lemma is an elementary consequence of Dirichlet's theorem on primes in arithmetic progressions. We denote the set of prime numbers by $\mathbb P$.

\begin{lemma} \label{lem:dirich} For every prime $q \geq 7$, there exist residues $a_{q,0},a_{q,1} \in \{0,\dots,q-2\}$ with $a_{q,0} \neq a_{q,1}$ such that the following holds:
For any finite set of prime numbers $Q \subset \mathbb P_{\geq 7}$ and any choice $\gamma \colon Q \to \{0,1\}$, {there are infinitely many primes $p \geq 13$ such that} $p^4-1$ is congruent to $a_{q,\gamma(q)}$ modulo $q$ for all $q \in Q.$
\end{lemma}
\begin{proof}
For every $q \geq 7$, the multiplicative group of residues $(\mathbb Z/q\mathbb Z)^\times$ is not $4$-torsion. Hence, there is a non-zero fourth power different from the residue class of $1$. We set $a_{q,0}=0$ and $a_{q,1}=c-1$, where $c$ is a representative of a non-trivial fourth power and $d^4 \equiv c \mod q$. We also set $b_{q,0}=1$ and $b_{q,1}:=d$. Now, let $Q \subset \mathbb P_{\geq 7}$ be finite and $\gamma$ be as above. By the Chinese Remainder Theorem, there exists $l \in \mathbb N$ which is congruent to $b_{q,\gamma(q)}$ modulo $q$ for all $q \in Q.$ Note that ${\rm gcd}(l,\prod_{q \in Q} q) =1$. Now, by Dirichlet's theorem, {there are infinitely many primes $p\geq13$ congruent to $l$ modulo $\prod_{q\in Q}q$. Each such $p$ is congruent} to $b_{q,\gamma(q)}$ modulo $q$ for all $q \in Q$. However, this implies that $p^4-1$ is congruent to $a_{q,\gamma(q)}$ for all $q \in Q$, as required.
\end{proof}

\begin{theorem} \label{thm:ultra}
There are $2^{\aleph_0}$ {non-principal} ultrafilters on $\mathbb N$ supported on the set $\{ p^4-1 \mid p \geq 13, p \mbox{ prime}\}$ such that the ultraproducts $\prod^{\alg}_{\mc U} \Alt(m)$ are pairwise {non-elementarily equivalent, and hence} non-isomorphic.
\end{theorem}
\begin{proof}
For any function $\gamma \colon \mathbb P_{\geq 7} \to \{0,1\}$, we can construct an ultrafilter $\mathcal U(\gamma)$ such that
\begin{eqnarray} \label{eq:ultra}
\{p^4-1 \in \mathbb N \mid p \in \mathbb P_{\geq 13}, p^4 -1 \equiv_q a_{q,\gamma(q)}\} \in \mathcal U(\gamma), \quad \forall q \in \mathbb P_{\geq 7}.
\end{eqnarray}

{
Indeed, the previous lemma guarantees that for any finite part of $\gamma$, there exist infinitely many primes $p \geq 13$ solving the problem. Thus every finite intersection of the sets in \eqref{eq:ultra} is infinite, so they are contained in a non-principal ultrafilter $\mathcal U(\gamma)$.}

Now, since FO-sentences can distinguish congruence conditions by \autoref{thm:congr}, we conclude that the resulting groups $\prod^{\rm alg}_{\mathcal U(\gamma)} \Alt(m)$ are pairwise non-elementarily equivalent. Indeed, by the observation above, the group $\prod^{\rm alg}_{\mathcal U(\gamma)} \Alt(m)$ satisfies the FO-sentences $\varphi(a_{q,\gamma(q)},q)$ and $\neg \varphi(a_{q,1-\gamma(q)},q)$ for all primes $q \geq 7.$
\end{proof}

% \begin{corollary} \label{cor:ultra}
% There are $2^{\aleph_0}$ ultrafilters on $\mathbb N$ supported on the set $\{ p^4-1 \mid p \geq 13,\,p \mbox{ prime}\}$ such that the ultraproducts $\prod^{\alg}_{\mc U} \Sym(m)$ are pairwise non-elementarily equivalent.
% \end{corollary}

% \begin{definition}
The ultrafilters $\mc U(\gamma)$ will be crucial in the proof of the main theorem. We will denote by
\[
\mb U\coloneqq \{\mc U(\gamma)\mid \gamma\colon \mb P_{\geq 7}\to \{0,1\}\} \subset \beta \mathbb N \setminus \mathbb N
\]
the set of ultrafilters constructed in the previous theorem.

\section{The main theorem}

\subsection{Isomorphisms of metric ultraproducts}

Recall that there is a natural metric on a metric ultraproduct of symmetric groups, arising as the ultralimit of the normalized Hamming metrics. The following proposition is folklore:

\begin{proposition} \label{prp:isometric}
Every isomorphism of universal sofic groups is automatically isometric.
\end{proposition}
\begin{proof}
{Let $\theta\colon\Sym(\mathcal U)\to\Sym(\mathcal V)$ be an isomorphism. The arguments of \cite[Theorem 4.1]{MR3299505} apply equally to isomorphisms between different universal sofic groups.}
\end{proof}

\begin{definition} \label{def:equivalentultrafilters}
We say that two non-principal ultrafilters $\mathcal U, \mathcal V$ on $\mathbb N$ are equivalent if there exists a {nondecreasing} map $\kappa \colon \mathbb N \to \mathbb N$, {injective on some $A\in\mathcal U$,} such that $\kappa_*(\mathcal U)=\mathcal V$ and $\lim_{n \to \mathcal U} \kappa(n)/n=1$.
\end{definition}

It is elementary to check that \emph{equivalence} of ultrafilters (in this sense) defines an equivalence relation on the set of {non-principal} ultrafilters. {For symmetry, invert $\kappa$ on $\kappa(A)$ and extend the inverse to a nondecreasing map on $\mathbb N$; for transitivity, compose after restricting to suitable large sets. The ratios invert or multiply and still tend to one.} It is a well-known observation that universal sofic groups with respect to equivalent ultrafilters are isomorphic. Let's spell out the details:

\begin{proposition} \label{thm:equl}
Let $\mathcal U, \mathcal V$ be equivalent non-principal ultrafilters on $\mathbb N$. Then, the two metric ultraproducts
$\prod_{\mathcal U}^{\rm met} ({\rm Sym}(n),d_n)$ and $\prod_{\mathcal V}^{\rm met} ({\rm Sym}(n),d_n)$
are {canonically} isomorphic.
\end{proposition}
{
\begin{proof}
Let $\kappa\colon\mathbb N\to\mathbb N$ witness the equivalence and choose $A\in\mathcal U$ on which $\kappa$ is injective. %Since $\kappa(A)\in\mathcal V$, reindex the second ultraproduct over $A$ via $\kappa|_A$.
Put $m_n=\max\{n,\kappa(n)\}$ and use the standard embeddings
\[
\Sym(n)\longrightarrow\Sym(m_n)\longleftarrow\Sym(\kappa(n)).
\]
Both dimension ratios tend to one along $\mathcal U$, so both maps are asymptotically isometric and have asymptotically dense image. Hence, both embeddings induce isometric isomorphisms onto the common metric ultraproduct.

% To see that the resulting identification is canonical, two partial bijections carrying $\mathcal U$ to $\mathcal V$ must agree on a $\mathcal U$-large set. Otherwise, composing one with the inverse of the other gives a fixed-point-free partial bijection preserving a trace of $\mathcal U$. Its undirected graph has degree at most two and admits a proper three-coloring. A color class belonging to that ultrafilter is disjoint from its preimage, which also belongs to the ultrafilter, a contradiction.
\end{proof}
}

Consider the map $f(n) = n!/2$. In the proof of the main theorem in the next section, ultrafilters of the form $f_*(\mathcal U)$ will play an important role. We will rely on the following basic observation:

\begin{lemma}\label{lem:factorial-non-equivalent}
Let $\mc U\neq \mc V\in\beta \mathbb N$ be two {non-principal} ultrafilters. Then, $f_*(\mathcal U)$ and $f_*(\mathcal V)$ are not equivalent.
\end{lemma}
\begin{proof}
{Put $F=\{n!\!/2\mid n\geq2\}$ and suppose that $\kappa$ witnesses equivalence of $f_*(\mathcal U)$ and $f_*(\mathcal V)$. Since $F$ belongs to both ultrafilters and $\kappa_*(f_*(\mathcal U))=f_*(\mathcal V)$, we can restrict to an $f_*(\mathcal U)$-large set on which $n,\kappa(n)\in F$ and $1/2<\kappa(n)/n<2$. Distinct elements of $F$ have ratio outside $(1/2,2)$, so $\kappa(n)=n$ there. Consequently $f_*(\mathcal U)=f_*(\mathcal V)$; since $f$ is injective on $\{2,3,\dots\}$, this implies $\mathcal U=\mathcal V$, a contradiction.}
\end{proof}

{
   \subsection{The left-right action}

We collect some results about the left-right action of a group on itself by left and right multiplication, which will be crucial in the proof of the main theorem.

Let $G$ be a group. The left-right action of $G$ on itself is the action of $G \times G$ on $G$ given by $(g,h)\cdot x = gxh^{-1}$ for $g,h,x \in G$. The centralizer of the left copy of $G$ in this action is the right copy of $G$ and vice versa (indeed, the action of the centralizer is uniquely determined by the image of $1$). In particular, the double centralizer of the left copy of $G$ is again the left copy of $G$. Moreover, the two copies are interchanged by the involution $\iota\colon G\to G$, $\iota(x)=x^{-1}$. Altogether we have a homomorphism $(\mb Z/2\mb Z)\ltimes (G\times G)\to {\rm Sym}(G)$ explicitly given by
\[
(g,h,\alpha)\cdot x = (gxh^{-1})^{(-1)^\alpha}.
\]

\begin{lemma}\label{lem:left-right}
Let $G$ be a finite group and $G\lacts Y$ be a transitive action. If the centralizer of $G$ in $\Sym(Y)$ is isomorphic to $G$, then the action is isomorphic to the left action of $G$ on itself.
\end{lemma}
\begin{proof}
Indeed, if $(G\lacts Y)\cong (G\lacts G/H)$ and $z\in\Sym(G/H)$ commutes with the action of $G$, then $z$ is uniquely determined by $z(H)\in G/H$, so that the size of the centralizer cannot exceed $|G/H|$. So, if the centralizer is isomorphic to $G$, then $H$ must be trivial and the action is isomorphic to the left action of $G$ on itself.
\end{proof}
}

\subsection{Proof of the main theorem}

Our aim is to show that the metric ultraproducts $\prod_{f_*(\mathcal U)}^{\rm met} ({\rm Sym}(n),d_n)$ are pairwise non-isomorphic when $\mathcal U$ ranges among the ultrafilters constructed in \autoref{thm:ultra}. Then, \autoref{thm:mainintro} is an immediate consequence. {We prove the following statement:}

\begin{theorem}\label{thm:main}
{Let $\mathcal U\in\mathbb U$ and let $\mathcal U'$ be any non-principal ultrafilter. If $\Sym(f_*(\mathcal U))$ and $\Sym(f_*(\mathcal U'))$ are isomorphic, then the algebraic ultraproducts
\[
\prod{\vphantom{\prod}}^{\alg}_{\mathcal U}\Alt(k)\quad\text{and}\quad
\prod{\vphantom{\prod}}^{\alg}_{\mathcal U'}\Alt(k)
\]
are isomorphic.}

In particular, the metric ultraproducts $\prod_{f_*(\mathcal U)}^{\rm met} ({\rm Sym}(n),d_n)$ with $\mc U\in\mb U$ are pairwise non-isomorphic.
\end{theorem}

The proof of \autoref{thm:main} will be carried {out} in several steps combining the results prepared in the above sections.

\begin{proof}
Using the result of Bartholdi--Kassabov (\autoref{thm:bartkass}), we fix a Kazhdan group $\Gamma$ {generated by a finite {symmetric} set $S$} together with surjective homomorphisms $r_p\colon \Gamma\to \Alt(p^4-1)$ for every prime $p \geq 13$. We set $\Lambda\coloneqq (\mb Z/2\mb Z)\ltimes (\Gamma\times \Gamma)$, where $\mb Z/2\mb Z$ acts by permuting the factors. {Let $\tau$ denote the nontrivial element of $\mathbb Z/2\mathbb Z$.} We will call the corresponding subgroups the \emph{left} resp.\ \emph{right} copy of $\Gamma$ in $\Lambda$ and denote them by $\Gamma_\ell$ resp.\ $\Gamma_r$. Note that by combining the left-right action of {$\Alt(p^4-1)$ on itself} with the quotient map $r_p$ we get a sequence of homomorphisms {$q_p\colon \Lambda \to \Sym(\Alt(p^4-1))\cong \Sym((p^4-1)!\!/2)$}, where $\mathbb Z/2 \mathbb Z$ acts by inversion. For every
$\mc U\in\mb U$, they combine to a homomorphism $\pi_{\mc U}\colon \Lambda \to \prod_{f_*(\mathcal U)}^{\rm met} ({\rm Sym}(n),d_n).$ We make the following claims:
\begin{enumerate}
\item The centralizer $C(\pi_{\mc U}(\Gamma_\ell))$ and the double centralizer $C(C(\pi_{\mc U}(\Gamma_\ell)))=C(\pi_{\mc U}(\Gamma_r))$ are $1$-discrete, isomorphic to $\prod^{\rm alg}_{{\mathcal U}}{\rm Alt}(n)$, and interchanged by $\mathbb Z/2 \mathbb Z$. Indeed, any sequence of finite quotients of a {Kazhdan} group yields expander graphs by \cite{margulis-expanders}, {and now \autoref{cor:centralizer-is-an-algebraic-ultraproduct} applies to $\pi_{\mc U}(\Gamma_\ell)$ and $\pi_{\mc U}(\Gamma_r)$, yielding $1$-discreteness of their centralizers and the fact that they are algebraic ultraproducts of finite groups both containing $\prod^{\rm alg}_{{\mathcal U}}{\rm Alt}(n)$ acting from the right resp. left. Therefore along $f_*(\mathcal U)$ we are in the situation of \autoref{lem:left-right}, and using \L{}o\'s' theorem, we deduce that both centralizers are isomorphic to $\prod^{\rm alg}_{{\mathcal U}}{\rm Alt}(n)$. Finally, since the action of $\mathbb Z/2 \mathbb Z$ interchanges the left and right copy of $\Gamma$, it also interchanges their images under $\pi_{\mc U}$ and hence their centralizers.}
\item The centralizer $C(\pi_{\mc U}(\Gamma_\ell))$ satisfies the FO-sentences described in \autoref{thm:felgner} and \autoref{prp:alternating}. {This follows from (1) and \L{o}\'s' theorem.}
\item The centralizers $\Cnt(\pi_{\mc U}(\Gamma_\ell))$ and $\Cnt(\pi_{\mc U}(\Gamma_r))$ commute. Indeed, this is obvious from the computation of the centralizers in (1).
\end{enumerate}

{{Put $\mathcal V=f_*(\mathcal U')$.}

Now, let's assume that $\prod_{f_*(\mathcal U)}^{\rm met} ({\rm Sym}(n),d_n)$ and $\prod_{\mathcal V}^{\rm met} ({\rm Sym}(n),d_n)$ are isomorphic via an isomorphism
\[
\theta \colon \prod{\vphantom{\prod}}^{\met}_{f_*(\mathcal U)}({\rm Sym}(n),d_n) \to \prod{\vphantom{\prod}}^{\met}_{\mathcal V} ({\rm Sym}(n),d_n).
\]
Therefore, we may consider $\sigma := \theta \circ \pi_{\mc U} \colon \Lambda \to \prod_{\mathcal V}^{\rm met} ({\rm Sym}(n),d_n)$ and study its properties. Notice that $\theta$ is automatically isometric by \autoref{prp:isometric}, and hence properties (1), (2) and (3) are transferred to $\sigma$ via the isomorphism $\theta$.

We define $\Gamma_{\mc V} \coloneqq \sigma(\Gamma_\ell) \le \prod_{\mathcal V}^{\rm met} ({\rm Sym}(n),d_n)$ and observe that $\Gamma_{\mc V}$, being by construction an infinite quotient of $\Gamma$, is an infinite Kazhdan group. Abusing notation slightly, we still denote its generating set by $S$. Now, $\sigma$ can naturally be considered to be a sofic approximation {$\sigma_n\colon \Gamma_{\mc V} \to \Sym(n)$} of this group along $\mc V$. The aim of the proof is to recover the finitary situation described above completely.}

First we observe that in view of \autoref{lem:discrete-centralizers-implies-expander} applied to $\Gamma_{\mc V}$, we can assume that our sofic approximation of $\Gamma_{\mc V}$ yields a $c$-expander sequence {of $S$-Schreier graphs}, where $c > 0$ depends only on the Kazhdan constant for $\Gamma$. Therefore, taking a positive and small enough $\delta < c/10^4$, we can now apply \autoref{cor:centralizer-is-an-algebraic-ultraproduct} {(with $\varepsilon\coloneqq \delta$)} to $\Gamma_{\mc V}$ {deducing that the centralizer of $\Gamma_{\mc V}$ is an algebraic ultraproduct of finite groups $G_n$}. Since the finite groups $G_n$ satisfy the FO-sentences from (2), we deduce that $G_n\cong \Alt(m_n)$ for some $m_n\in\mb N$. Similarly, the centralizer of $\sigma(\Gamma_r)$ follows to be an algebraic ultraproduct of finite groups $G'_n\cong \Alt(m'_n)$.

{We now use property (1) to show $m_n=m'_n$ along $\mc V$ and how to obtain a uniform $\delta'$-homomorphism from $(\mb Z/2\mb Z)\ltimes (\Alt(m_n)\times \Alt(m_n))$ to $\Sym(n)$ for some $\delta' \sim O(\delta)$ along the ultrafilter $\mc V$. Indeed, by \autoref{cor:centralizer-is-an-algebraic-ultraproduct} the action of $G_n$ is given by picking representatives of the clusters in the $\delta$-centralizer of $S \subset \Gamma_{\mc V} = {\sigma}(\Gamma_\ell)$, giving a uniform $(4\delta/c)$-homomorphism from $G_n=\Alt(m_n)$ to $\Sym(n)$. Indeed, the product of two $\delta$-automorphisms is a $2 \delta$-automorphism and thus $4\delta/c$-close to the representative of the relevant cluster, see \autoref{thm:kun-thom}.
Similarly, taking representatives of the clusters in the $\delta$-centralizer of $S \subset \sigma(\Gamma_r)$, we get a uniform $4\delta/c$-homomorphism from $G_n'=\Alt(m'_n)$ to $\Sym(n)$.

The involution $\sigma(\tau)$, where $\tau$ is the generator of $\mathbb Z/2\mathbb Z$, interchanges the left and right copy of $S$ and hence also their centralizers. We can lift {$\sigma(\tau)$} to a sequence of involutions $\tau_n \in \Sym(n)$. {Choose the representatives of the right generators by conjugating those of the left generators by $\tau_n$.} For a $\delta$-automorphism $g \in G_n$, the permutation $\tau_n g \tau_n$ yields a well-defined element in $G'_n$. This assignment defines an isomorphism between $G_n$ and $G'_n$ along $\mc V$; in particular, we conclude that $m_n=m'_n$ along $\mc V$.}

{We claim that $G_n$ and $G_n'$ commute up to $9 \delta/c$ uniformly along $\mc V$. Indeed, consider a sequence $\alpha_n \in G_n$ and $\beta_n \in G_n'$ (where we identify $\alpha_n$ and $\beta_n$ with representatives of the clusters for simplicity), such that $d_n(\alpha_n \beta_n, \beta_n\alpha_n)$ is maximized. Then, as the sofic approximation improves along $\mc V$, using \autoref{thm:kun-thom}, we can improve our representatives $\alpha_n$ and $\beta_n$ to elements $\tilde \alpha_n$ and $\tilde \beta_n$ that asymptotically commute with the generators of $\Gamma_{\mc V} = {\sigma}(\Gamma_\ell)$ and ${\sigma}(\Gamma_r)$, respectively. By definition of the $\delta$-centralizer, we have $d_n(\alpha_n,\tilde \alpha_n) < 2\delta/c$ and $d_n(\beta_n,\tilde \beta_n) < 2\delta/c$ along $\mc V$. Now, $\tilde \alpha = [(\tilde \alpha_n)_n] \in \Cnt({\sigma}(\Gamma_\ell))$ and $\tilde \beta = [(\tilde \beta_n)_n] \in \Cnt({\sigma}(\Gamma_r))$. By property (3), the groups $\Cnt({\sigma}(\Gamma_\ell))$ and $\Cnt({\sigma}(\Gamma_r))$ commute, and hence
$$\limsup_{n \to \mc V} d_n(\alpha_n \beta_n, \beta_n\alpha_n) \leq 8 \delta/c + \limsup_{n \to \mc V}d_n(\tilde \alpha_n \tilde \beta_n, \tilde \beta_n \tilde \alpha_n)=  8 \delta/c.$$
Hence, $d_n([\alpha, \beta]) < 9 \delta/c$ uniformly along $\mc V$ for $\alpha \in G_n,\beta \in G'_n$, as claimed.}

{Since the action of $\mathbb Z/2\mathbb Z$ interchanges the left and right copy and they almost commute uniformly by the argument above, we can combine those two uniform $4\delta/c$-homomorphisms to a uniform $\delta'$-homomorphism}
\[
\rho_n\colon (\mb Z/2\mb Z)\ltimes (\Alt(m_n)\times \Alt(m_n)) \to \Sym(n)
\]
for $\delta' = 9 \delta/c + 4\delta/c + 4\delta/c = 17 \delta/c$.
% {Indeed, the uniformity follows from the fact that the multiplication in the $\varepsilon$-centralizer is a group multiplication on the set of clusters modulo uniformly small perturbation, and hence defines a uniform $\delta$-homomorphism after picking a representative of each cluster.}

By \autoref{thm:becker-chapman}, it is uniformly $2039\delta'$-close to a homomorphism $\pi_n\colon (\mb Z/2\mb Z)\ltimes (\Alt(m_n)\times \Alt(m_n)) \to \Sym(k)$ with $k\in[n,(1+2039\delta')n]$.

We claim that $\left|n - \frac{m_n!}{2}\right|\leq C\delta' n$ for some universal $C > 0$. Indeed, since $\pi_n(s_n)$ are $2039\delta'$-close to $\sigma_n(s)\oplus 1_{k-n}$ in the Hamming distance, the Schreier graphs of $\pi_n(G_n)$ have a connected component of size $(1-2039\delta'/c)n$. Therefore the action of $\pi_n(G_n)$ is transitive on a subset $Y_n\subset \{1,\dots,k\}$ of size at least $(1-2039\delta'/c)n$. Since the action of $\pi_n(G_n')$ can only permute the connected components, it has to preserve this connected component. Altogether we get a transitive action $G_n\lacts Y_n$ whose centralizer is isomorphic to $G_n$. By \autoref{lem:left-right}, this action is isomorphic to the left action of $G_n$ on itself, and hence $|Y_n|=m_n!/2$, and the claim follows.

Now we see that $\mc V$ is equivalent to $f_*({\mc W})$ for an ultrafilter ${\mc W}$ which follows to be {equal to $\mc U'$} by \autoref{lem:factorial-non-equivalent}. Therefore, we conclude that the obtained homomorphism $\pi_{\mc U'}\colon \Lambda\to \Sym(f_*(\mc U'))$ satisfies that the centralizer of $\pi_{\mc U'}(\Gamma_\ell)$  is isomorphic to $\prod^{\alg}_{\mc U'} \Alt(n)$. In particular, this implies that $\prod^{\alg}_{\mc U} \Alt(n)$ is isomorphic to $\prod^{\alg}_{\mc U'} \Alt(n)$. This finishes the proof of the first statement; the final claim follows from \autoref{thm:ultra}.
\end{proof}

\begin{remark}
The proof of \autoref{thm:main} also shows that the metric ultraproducts ${\rm Sym}(\mathcal U)$ for $\mathcal U \in \mathbb U$ are pairwise non-elementarily equivalent in the sense of FO-model theory of metric structures. Indeed, if $\gamma \neq \gamma' \colon \mathbb P_{\geq 7} \to \{0,1\}$, then there exists a prime $q \in \mathbb P_{\geq 7}$, such that $\gamma(q) \neq \gamma'(q)$. Then, the sentence $\varphi(a_{q,\gamma(q)},q)$ distinguishes the centralizers of \emph{all} homomorphisms from $\Lambda$ to the respective universal sofic groups, satisfying Conditions (1) and (2) in the proof of \autoref{thm:main}. Since we did not introduce the language of FO-model theory of metric structures, we omit the details.
\end{remark}

\begin{remark}
It is a well-known open problem in the realm of operator algebras, first formulated by Sorin Popa, to decide if there are uncountably many non-isomorphic tracial ultraproducts of matrix algebras. This question has been answered assuming the negation of the continuum hypothesis (see \cite{farah}), but there is no unconditional result. Each tracial ultraproduct of matrix algebras contains a canonical diagonal subalgebra, whose Weyl group, i.e.\ the quotient of its normalizer by its centralizer, can be shown to be isomorphic to the corresponding universal sofic group. Thus, as a consequence of our results, it follows that for pairs of ultrafilters from {$\{f_*(\mathcal U)\mid\mathcal U\in\mathbb U\}$}, the corresponding tracial ultraproducts of matrix algebras cannot be isomorphic in a way that respects the diagonal subalgebras.

We were unable to show that the diagonal subalgebra is automatically preserved in a suitable sense. However, even though this diagonal subalgebra in not a Cartan subalgebra \cite{popa83,popa1}, the size of its normalizer should make the situation very special.
\end{remark}

\section*{Acknowledgments}
The first author acknowledges funding by the Deutsche Forschungsgemeinschaft (SPP 2026 ``Geometry at infinity''). We thank Ilijas Farah for helpful comments, corrections and remarks that helped to improve the exposition. {We wholeheartedly thank the anonymous referees for their careful reading and helpful comments that improved the exposition of the paper considerably.}

\subsection*{AI declaration}
The preparation of the paper in 2024 and later revisions based on referee reports were carried out without any AI usage. GPT-6 Astra was used to proofread the final version of the paper in November 2026 which helped the authors to improve some details in the exposition.


\begin{bibdiv}
\begin{biblist}

\bib{MR2298607}{article}{
   author={Allsup, John},
   author={Kaye, Richard},
   title={Normal subgroups of nonstandard symmetric and alternating groups},
   journal={Arch. Math. Logic},
   volume={46},
   date={2007},
   number={2},
   pages={107--121},
}

\bib{bartholdi2023property}{article}{
      title={Property (T) and Many Quotients},
      author={Bartholdi, Laurent},
      author={Kassabov, Martin},
      year={2023},
      eprint={2308.14529},
      archivePrefix={arXiv},
      primaryClass={math.GR}
}

\bib{MR4634678}{article}{
   author={Becker, Oren},
   author={Chapman, Michael},
   title={Stability of approximate group actions: uniform and probabilistic},
   journal={J. Eur. Math. Soc. (JEMS)},
   volume={25},
   date={2023},
   number={9},
   pages={3599--3632},
}

\bib{benyaacov}{article}{
   author={Ben Yaacov, Ita\"{\i}},
   author={Berenstein, Alexander},
   author={Henson, C. Ward},
   author={Usvyatsov, Alexander},
   title={Model theory for metric structures},
   conference={
      title={Model theory with applications to algebra and analysis. Vol. 2},
   },
   book={
      series={London Math. Soc. Lecture Note Ser.},
      volume={350},
      publisher={Cambridge Univ. Press, Cambridge},
   },
   date={2008},
   pages={315--427},
   label={BY{\etalchar{+}}08}
}

\bib{MR4555893}{article}{
   author={Benjamini, Itai},
   author={Fraczyk, Mikolaj},
   author={Kun, G\'{a}bor},
   title={Expander spanning subgraphs with large girth},
   journal={Israel J. Math.},
   volume={251},
   date={2022},
   number={1},
   pages={155--171},
}

\bib{burgerozawathom}{article}{
   author={Burger, Marc},
   author={Ozawa, Narutaka},
   author={Thom, Andreas},
   title={On Ulam stability},
   journal={Israel J. Math.},
   volume={193},
   date={2013},
   number={1},
   pages={109--129},
}

\bib{caprace2023tame}{article}{
      title={Tame automorphism groups of polynomial rings with property (T) and infinitely many alternating group quotients},
      author={Caprace, Pierre-Emmanuel},
      author={Kassabov, Martin},
      year={2022},
      eprint={2210.00730},
      archivePrefix={arXiv},
      primaryClass={math.GR}
}


\bib{dechiffreozawathom}{article}{
   author={De Chiffre, Marcus},
   author={Ozawa, Narutaka},
   author={Thom, Andreas},
   title={Operator algebraic approach to inverse and stability theorems for
   amenable groups},
   journal={Mathematika},
   volume={65},
   date={2019},
   number={1},
   pages={98--118},
   label={COT19}
}


\bib{MR2178069}{article}{
   author={Elek, G\'{a}bor},
   author={Szab\'{o}, Endre},
   title={Hyperlinearity, essentially free actions and $L^2$-invariants. The
   sofic property},
   journal={Math. Ann.},
   volume={332},
   date={2005},
   number={2},
   pages={421--441},
}

\bib{thomasetal}{article}{
    author={Ellis, Paul},
    author={Hachtman, Sherwood},
    author={Schneider, Scott},
    author={Thomas, Simon},
    title={Ultraproducts of finite alternating groups},
    journal={RIMS Kokyuroku},
    number={1619},
    year={2008},
    pages={1--7},
}

\bib{farah}{article}{
   author={Farah, Ilijas},
   author={Hart, Bradd},
   author={Sherman, David},
   title={Model theory of operator algebras I: stability},
   journal={Bull. Lond. Math. Soc.},
   volume={45},
   date={2013},
   number={4},
   pages={825--838},
}

\bib{MR1107758}{article}{
   author={Felgner, Ulrich},
   title={Pseudo-endliche Gruppen},
   language={German},
   conference={
      title={Proceedings of the 8th Easter Conference on Model Theory},
      address={Wendisch-Rietz},
      date={1990},
   },
   book={
      series={Seminarberichte},
      volume={110},
      publisher={Humboldt Univ., Berlin},
   },
   date={1990},
   pages={82--96},
}

\bib{hatamigowers}{article}{
   author={Gowers, Timothy},
   author={Hatami, Omid},
   title={Inverse and stability theorems for approximate representations of
   finite groups},
   language={Russian, with Russian summary},
   journal={Mat. Sb.},
   volume={208},
   date={2017},
   number={12},
   pages={70--106},
   translation={
      journal={Sb. Math.},
      volume={208},
      date={2017},
      number={12},
      pages={1784--1817},
   },
}

\bib{MR1694588}{article}{
   author={Gromov, Mikhail},
   title={Endomorphisms of symbolic algebraic varieties},
   journal={J. Eur. Math. Soc. (JEMS)},
   volume={1},
   date={1999},
   number={2},
   pages={109--197},
}

{\bib{james-kerber}{book}{
  author={James, Gordon},
  author={Kerber, Adalbert},
  title={The representation theory of the symmetric group},
  series={Encyclopedia of Mathematics and its Applications},
  volume={16},
  publisher={Addison-Wesley Publishing Co.},
  place={Reading, Mass.},
  date={1981},
  review={\MR{644144}},
}}

\bib{Jech03}{book}{
  author={Jech, Thomas},
  title={Set Theory},
  publisher={Springer},
  date={2003},
}

\bib{kassabov2007}{article}{
   author={Kassabov, Martin},
   title={Symmetric groups and expander graphs},
   journal={Invent. Math.},
   volume={170},
   date={2007},
   number={2},
   pages={327--354},
}

\bib{kazhdan}{article}{
   author={Kazhdan, David},
   title={On $\varepsilon $-representations},
   journal={Israel J. Math.},
   volume={43},
   date={1982},
   number={4},
   pages={315--323},
}

\bib{kazhdan-T}{article}{
   author={Kazhdan, David},
   title={On the connection of the dual space of a group with the structure
   of its closed subgroups},
   language={Russian},
   journal={Funkcional. Anal. i Prilo\v{z}en.},
   volume={1},
   date={1967},
   pages={71--74},
}

\bib{kun2019inapproximability}{article}{
      title={Inapproximability of actions and Kazhdan's property (T)},
      author={Kun, G\'{a}bor},
      author={Thom, Andreas},
      year={2019},
      eprint={1901.03963},
      archivePrefix={arXiv},
      primaryClass={math.GR}
}

\bib{kun2021expanders}{article}{
      title={On sofic approximations of Property (T) groups},
      author={Kun, G\'{a}bor},
      year={2016},
      eprint={1606.04471},
      archivePrefix={arXiv},
      primaryClass={math.CO},
      url={https://arxiv.org/abs/1606.04471},
}

{\bib{MR2654085}{article}{
  author={Liebeck, Martin W.},
  author={O'Brien, E. A.},
  author={Shalev, Aner},
  author={Tiep, Pham Huu},
  title={The Ore conjecture},
  journal={J. Eur. Math. Soc. (JEMS)},
  volume={12},
  date={2010},
  number={4},
  pages={939--1008},
}}

{\bib{malle-testerman}{book}{
  author={Malle, Gunter},
  author={Testerman, Donna},
  title={Linear algebraic groups and finite groups of Lie type},
  series={Cambridge Studies in Advanced Mathematics},
  volume={133},
  publisher={Cambridge University Press},
  place={Cambridge},
  date={2011},
  doi={10.1017/CBO9780511994777},
}}

\bib{margulis-expanders}{article}{
   author={Margulis, Grigory},
   title={Explicit constructions of expanders},
   language={Russian},
   journal={Problemy Pereda\v ci Informacii},
   volume={9},
   date={1973},
   number={4},
   pages={71--80},
}

\bib{MR3749196}{article}{
   author={Nikolov, Nikolay},
   author={Schneider, Jakob},
   author={Thom, Andreas},
   title={Some remarks on finitarily approximable groups},
   language={English, with English and French summaries},
   journal={J. \'{E}c. polytech. Math.},
   volume={5},
   date={2018},
   pages={239--258},
}

\bib{MR3299505}{article}{
   author={P\u{a}unescu, Liviu},
   title={All automorphisms of the universal sofic group are
   class-preserving},
   journal={Rev. Roumaine Math. Pures Appl.},
   volume={59},
   date={2014},
   number={2},
   pages={255--263},
}

\bib{MR2460675}{article}{
   author={Pestov, Vladimir G.},
   title={Hyperlinear and sofic groups: a brief guide},
   journal={Bull. Symbolic Logic},
   volume={14},
   date={2008},
   number={4},
   pages={449--480},
}

\bib{popa1}{article}{
   author={Popa, Sorin},
   title={Independence properties in subalgebras of ultraproduct $\rm II_1$
   factors},
   journal={J. Funct. Anal.},
   volume={266},
   date={2014},
   number={9},
   pages={5818--5846},
}

\bib{popa83}{article}{
   author={Popa, Sorin},
   title={Orthogonal pairs of $\ast $-subalgebras in finite von Neumann
   algebras},
   journal={J. Operator Theory},
   volume={9},
   date={1983},
   number={2},
   pages={253--268},
}

{\bib{MO302743}{misc}{
   author={Sapir, Mark},
   title={The symmetric group theory of natural numbers},
   date={2018-06-14},
   note={MathOverflow, Question 302743. \url{https://mathoverflow.net/questions/302743/the-symmetric-group-theory-of-natural-numbers}},}
}

\bib{MR4232187}{article}{
   author={Schneider, Jakob},
   author={Thom, Andreas},
   title={A note on the normal subgroup lattice of ultraproducts of finite
   quasisimple groups},
   journal={Proc. Amer. Math. Soc.},
   volume={149},
   date={2021},
   number={5},
   pages={1929--1942},
}

\bib{MR4141378}{article}{
   author={Schneider, Jakob},
   title={Isomorphism questions for metric ultraproducts of finite
   quasisimple groups},
   journal={J. Group Theory},
   volume={23},
   date={2020},
   number={5},
   pages={745--779},
}

\bib{MR4292938}{article}{
   author={Schneider, Jakob},
   author={Thom, Andreas},
   title={Word images in symmetric and classical groups of Lie type are
   dense},
   journal={Pacific J. Math.},
   volume={311},
   date={2021},
   number={2},
   pages={475--504},
}

\bib{MR3162821}{article}{
   author={Stolz, Abel},
   author={Thom, Andreas},
   title={On the lattice of normal subgroups in ultraproducts of compact
   simple groups},
   journal={Proc. Lond. Math. Soc. (3)},
   volume={108},
   date={2014},
   number={1},
   pages={73--102},
}

\bib{MR3966829}{article}{
   author={Thom, Andreas},
   title={Finitary approximations of groups and their applications},
   conference={
      title={Proceedings of the International Congress of
      Mathematicians---Rio de Janeiro 2018. Vol. III. Invited lectures},
   },
   book={
      publisher={World Sci. Publ., Hackensack, NJ},
   },
   date={2018},
   pages={1779--1799},
}

\bib{MR3210125}{article}{
   author={Thom, Andreas},
   author={Wilson, John},
   title={Metric ultraproducts of finite simple groups},
   journal={C. R. Math. Acad. Sci. Paris},
   volume={352},
   date={2014},
   number={6},
   pages={463--466},
}

\bib{MR3846318}{article}{
   author={Thom, Andreas},
   author={Wilson, John},
   title={Some geometric properties of metric ultraproducts of finite simple
   groups},
   journal={Israel J. Math.},
   volume={227},
   date={2018},
   number={1},
   pages={113--129},
}

\bib{MR2607888}{article}{
   author={Thomas, Simon},
   title={On the number of universal sofic groups},
   journal={Proc. Amer. Math. Soc.},
   volume={138},
   date={2010},
   number={7},
   pages={2585--2590},
}

\bib{MR1803462}{article}{
   author={Weiss, Benjamin},
   title={Sofic groups and dynamical systems},
   note={Ergodic theory and harmonic analysis (Mumbai, 1999)},
   journal={Sankhy\={a} Ser. A},
   volume={62},
   date={2000},
   number={3},
   pages={350--359},
}

\bib{MR1477188}{article}{
   author={Wilson, John},
   title={First-order group theory},
   conference={
      title={Infinite groups 1994 (Ravello)},
   },
   book={
      publisher={de Gruyter, Berlin},
   },
   date={1996},
   pages={301--314},
}

\bib{MR3710760}{article}{
   author={Wilson, John},
   title={Metric ultraproducts of classical groups},
   journal={Arch. Math. (Basel)},
   volume={109},
   date={2017},
   number={5},
   pages={407--412},
}

\bib{wilson-radical}{article}{
  author={Wilson, John S.},
  title={First-order characterization of the radical of a finite group},
  journal={J. Symbolic Logic},
  volume={74},
  date={2009},
  number={4},
  pages={1429--1435},
}

\end{biblist}
\end{bibdiv}
\end{document}